\documentclass{llncs}
\usepackage[utf8]{inputenc}
\usepackage{enumerate}
\usepackage{graphicx}
\usepackage{epstopdf}
\usepackage{amsfonts}
\usepackage{amssymb}
\usepackage{amsmath}
\usepackage[labelfont=bf]{caption}
\usepackage{sectsty}
\usepackage{wasysym}
\usepackage{hyperref}
\usepackage{parskip}

\newcommand{\N}{\mathbb{N}}

\begin{document}

\title{Edge-sum distinguishing labeling}
\author{Jan Bok \inst{1} \and Nikola Jedličková \inst{2}}
\institute{
Computer Science Institute, Faculty of Mathematics and Physics, Charles University, Malostransk\'{e} n\'{a}m\v{e}st\'{i} 25, 11800, Prague, Czech Republic. Email: \email{bok@iuuk.mff.cuni.cz}
\and
Department of Applied Mathematics, Faculty of Mathematics and Physics, Charles University, Malostransk\'{e} n\'{a}m\v{e}st\'{i} 25, 11800, Prague, Czech Republic. Email: \email{jedlickova@kam.mff.cuni.cz}
}

\maketitle

\begin{abstract}
We study \emph{edge-sum distinguishing labeling}, a type of labeling
recently introduced by Tuza in [Zs. Tuza, \textit{Electronic Notes in Discrete Mathematics} 60, (2017), 61-68] in context of labeling games.

An \emph{ESD labeling} of an $n$-vertex graph $G$ is an injective mapping of
integers $1$ to $l$ to its vertices such that for every edge, the sum of the
integers on its endpoints is unique. If $l$ equals to $n$, we speak about a
\emph{canonical ESD labeling}.

We focus primarily on structural properties of this labeling and show for several
classes of graphs if they have or do not have a canonical ESD labeling.
As an application we show some implications of these results for games based
on ESD labeling. We also observe that ESD labeling is closely connected to the well-known
notion of \emph{magic} and \emph{antimagic} labelings, to the \emph{Sidon sequences} and to \emph{harmonious labelings}.
\end{abstract}

\section{Introduction and preliminaries}

Graph labeling is a vivid area of combinatorics which started in the middle of
1960's. Much of the area is based on results of Rosa \cite{rosa1967} and of
Graham and Sloane \cite{graham1980additive}. Since then, over 200 different
labelings were introduced. We refer to Gallian's survey \cite{dynamic}, citing
over 2500 papers, gathering most of the results in the area. Applications of
labeling are both theoretical (Rosa introduced so-called \emph{graceful}
labelings to attack Ringel's conjecture on certain graph decompositions) and
practical (for example the frequency assignment problem
\cite{hale1980,vangraph,jha2000}).

We study \emph{edge-sum distinguishing} (abbreviated as ESD) labeling,
introduced by Tuza \cite{tuza2017graph} in 2017. Tuza's primarily concern was
to study several combinatorial games connected to this labeling. Our main
objective is to study structural properties of this labeling on its own.
However, as our secondary objective, we also give some results on game
variants of edge-sum distinguishing labeling.

\paragraph{Structure of the paper.} In the rest of this section we review
basic definitions and show a broader context of ESD labeling to other existing
notions in combinatorics. The second section deals with structural properties
of ESD labeling. For various well-known classes of graphs we show if they have
a canonical ESD labeling or not. In the third section we are concerned with
game variants, the original motivation of Tuza. Finally, in the last section
we summarize our results and propose some open problems.

\paragraph{Notation.}

We use the notation of West \cite{west2001introduction}. All graphs in the
paper are finite, undirected, connected and without multiple edges, unless
we say otherwise.

\subsection{Basic definitions}

We need to formally define what graph labeling is. We will need vertex
labelings only.

\begin{definition}
Let $G = (V,E)$ be a graph and let $L \subseteq \mathbb{N}$ be a set of labels. Then a mapping $\phi: V \rightarrow L$ is called a \emph{vertex labeling}.
We further say that vertex labeling is \emph{canonical} if $|V| = |L|$.
\end{definition}

We will often refer to edge-weights, induced by a vertex labeling.

\begin{definition}
Let $G=(V,E)$ be a graph and $\phi$ a vertex labeling on $G$. The \emph{edge-weight}
of an edge $xy$ is defined as $w_\phi(xy):=\phi(x) + \phi(y)$.
\end{definition}

Now we can finally introduce a definition of \emph{edge-sum distinguishing
labeling}.

\begin{definition}
Let $G=(V,E)$ be a graph and $L = \{1, \ldots, l\}$, $l \in \mathbb N$. A vertex labeling $\phi: V \rightarrow L $ is called \emph{edge-sum distinguishing labeling}
(\emph{ESD labeling}) if $\phi$ is injective and if
$$\forall e, f \in E: e \neq f \rightarrow w_\phi(e) \neq w_\phi(f).$$
\end{definition}

We note that no ESD labeling exists in case $|L| < |V|$. We call a special case
when $|L| = |V|$ a \emph{canonical ESD labeling}.

\begin{example}
Consider a path $P_n$ and denote its vertices consecutively $v_1,\ldots,v_n$.
Choose a labeling $\phi(v_i) = i$. Clearly, this is an ESD labeling and even a canonical ESD labeling.
\end{example}

\subsection{Connections to existing notions}

\paragraph{Edge-antimagic vertex labeling.} 

Following the usual terminology in the area of graph labelings, one could name
canonical ESD labelings also as \emph{edge-antimagic vertex labelings}. To
illustrate this, let us recall that an \emph{antimagic labeling} of a graph
with $m$ edges and $n$ vertices is a bijection from the set of edges to the
integers $1,\ldots,m$ such that all $n$ vertex sums are pairwise distinct,
where a vertex sum is the sum of labels of all edges incident with the same
vertex. Antimagic labeling were introduced as a natural generalization of
magic labelings. We refer the reader to
\cite{bavca2007edge,simanjuntak2000two} for more information on antimagic
labelings and to \cite{kotzig1970magic,kotzig1972magic,wallis2001magic} for a
literature on magic labelings.

To our best knowledge, edge-antimagic vertex labelings were not studied yet.

\paragraph{Super edge-magic total labelings.} 

A \emph{super edge-magic total labeling} is an injection $f: V \cup E \to
\{1, 2, \ldots , |V| + |E|\}$ such that the weight of every edge $xy$ defined
as $w(xy) = f (x) + f (y) + f (xy)$ is equal to the same magic constant $m$
and the vertex labels are the numbers $1, 2, \ldots , |V|$. One can observe
that such labeling implies an edge-sum distinguishing labeling in a natural
way. If we remove the labels of edges, the edge-weights now form an arithmetic
progression. We can say about the resulting labeling even more; it is an
$(a,1)$-edge antimagic vertex labeling.

An \emph{$(a,d)$-edge antimagic vertex labeling} is a one-to-one mapping
$f$ from $V(G)$ onto $\{1,2,\ldots,|V| \}$ with the property that
for every edge $xy \in E(G)$, the edge-weight set is equal to
$$\{f(x)+f(y): x,y\in V(G) \} = \{a,a+d,a+2d,\ldots,a+(|E|+1)d \},$$
for some $a>0, d\geq 0$. This definition comes from \cite{sugeng2013construction}.

\paragraph{Sidon sequences.}  The Sidon sequences were introduced by Simon Sidon
in 1932 \cite{sidon1932satz}. We refer the reader to a dynamically 
updated survey of O’Bryant \cite{o2004complete}.
The formulation of the following definition comes from the survey.

\begin{definition}
	A \emph{Sidon sequence} is a sequence of integers $a_1 < a_2 < \ldots$ with
	the property that sums $a_i+a_j \,\, (i \leq j),$ are distinct.
\end{definition}

ESD labeling can be reformulated in a similar fashion.

\begin{definition}
	An \emph{ESD labeling} of a graph $G=(V,E)$, where $V=\{1,\ldots,n\}$, is a sequence of
	integers $a_1 < a_2 < \ldots$ with the property that
	sums $a_i+a_j, \, i \leq j,\, (i,j) \in E,$ are distinct and $a_1 = 1$.
\end{definition}

With this new definition in hand, we see that ESD labeling
is in some sense a generalization of the Sidon sequence.
The difference that $a_1 = 1$ in the definition of ESD labeling could be
easily dropped (but it is convenient for this paper). Also, one can observe
that without this condition, the original Sidon sequences are ESD labelings 
of a sufficiently large complete graph with loops added to each vertex.
However, again for our convenience, we consider only loopless graphs in
this paper.

\paragraph{Harmonious labeling.}
\emph{Harmonious labeling} was introduced by Graham and Sloane \cite{graham1980additive}.
We say that graph $G$ with $k$ edges is \emph{harmonious} if its vertices
can be labeled injectively with integers modulo $k$ so that the sum of the labels
of its endpoints modulo $k$ is unique.

The difference between harmonious and ESD labeling is that we do not take
vertex labels and edge labels modulo number of edges in ESD case. In fact, ESD labelings
and harmonious labelings behave differently. For example, it is
conjectured that trees are harmonious and it is known that not all cycles are
harmonious \cite{guy2013unsolved}. For comparison, we show that
all trees and cycles have a canonical ESD labeling.

\section{Structural results}

\subsection{Necessary condition}

\begin{theorem} \label{thm:nec}
If a graph $G=(V,E)$ such that $|V| > 1$ has a canonical ESD labeling, then the inequality $|E| \leq 2|V|-3$ holds.
\end{theorem}
\begin{proof}
	We claim that every canonical ESD labeling of an $n$-vertex graph has at most
	$2n-3$ different edge-weights.

	To prove this, observe that the smallest possible edge-weight in such
	labeling is 3 and
	the largest possible is $2n-1$. Also, the edge-weights  of $G$ form a subset of the
	set $\{3,\ldots,2n-1 \}$ which is of the size $2n-3$. This proves the claim.

	Now if a graph $G$ has more than $2|V|-3$ edges we can use our claim
	and by the pigeonhole principle, we have two edges with the same
	weight, a contradiction.
\qed
\end{proof}

Now we will show that this bound is tight.

\begin{theorem}
	For every $n \in \N, n > 1,$ there exist an $n$-vertex graph $G_n$ with $|E(G_n)| = 2n - 3$ which has a canonical ESD labeling.
\end{theorem}
\begin{proof}
	For $G_2$ take $K_2$ and for $G_3$ take $K_3$. These cases are trivial.

	For $n>3$, take a complete bipartite graph $K_{2,n-2}$ and add
	an edge between the two vertices of the part of size 2. See Figure \ref{fig:tight}
	for an example.

\begin{figure}
\centering
\includegraphics[scale=0.8]{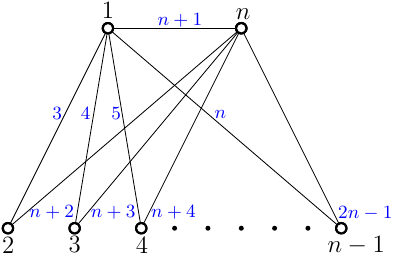}
\caption{An example of an ESD graph with $2n-3$ edges.}
\label{fig:tight}
\end{figure} 

	We will show that this graph has a canonical ESD labeling. We will denote
	$x_1,x_2$ the vertices of the part of size 2 and $y_1,\ldots,y_{n-2}$ the vertices
	of the other part.

	Now we define a labeling $\phi$ in the following way.
	\begin{itemize}
		\item Let $\phi(x_1) = 1$ and $\phi(x_2) = n$.
		\item Let $\phi(y_i) = i+1$ for $1 \leq i \leq n-2$.
	\end{itemize}

	Observe that the edges incident with $x_1$ have edge-weights from $3$ to $n+1$.
	Furthermore, the edges incident with $x_2$, except for the edge $x_1x_2$,
	have edge-weights ranging from $n+2$ to $2n-1$. All these weights appear exactly once
	and thus we are done.
\qed
\end{proof}

\subsection{Fan graphs}

In the previous part we showed a necessary condition for graph to have a
canonical ESD labeling. The point of this part is to show that this condition
is not sufficient in general by proving that \emph{fan graphs}, which have
$2n-3$ edges, do not have a canonical ESD labeling if their order is bigger
than 8.

\begin{definition}
	A \emph{fan graph} $F_n$ is a path $P_{n-1}$ and one other vertex $v$ (we call
	it the \emph{central vertex}) joined by an edge with every vertex of the path.
	See Figure \ref{fig:vejar} for an example.
\end{definition}

\begin{figure}
\centering
\includegraphics[scale=0.8]{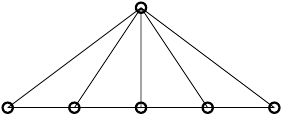}
\caption{A fan graph $F_6$.}
\label{fig:vejar}
\end{figure} 

\begin{theorem} \label{thm:fan}
	A fan graph $F_n$ does not have a canonical ESD labeling if and only if $n \ge 8$.
\end{theorem}
\begin{proof}
Note that $F_n$ for $n$ up to 7 has a canonical ESD labeling, as we can see on Figure \ref{fig:vejar_male}. It is obvious that $F_2$ and $F_3$ have canonical ESD labelings.

From Theorem \ref{thm:nec} we know that we have at most $2n-3$ different
edge-weights. Since a fan graph of order $n$ has exactly $2n-3$ edges we need
to use every possible edge-weight from the set $\{3,\ldots,2n-1 \}$ exactly once.

The edge-weights 3 and 4 can be obtained in exactly one possible way. In the first
case on an edge with endpoints labeled 1 and 2, in the second case on an edge
with endpoints 1 and 3. The edge-weight 5 can be obtained in two ways. Either
as the weight of an edge with endpoints $2$ and $3$ or as the weight
on an edge with endpoints $1$ and $4$. We get two possible subgraphs $S_1$ and
$S_2$.

By a similar analysis, one can get the labeled subgraphs $S_3$ and $S_4$.

Hence, exactly one of the labeled subgraphs $S_1$ or $S_2$ has to be in $F_n$
and, analogously, one of the $S_3$ and $S_4$ as well.
However, in all graphs $S_i, i \in \{1,\ldots,4 \}$, one of its vertices has
to be the central vertex. Since $n \ge 8$, we see that the minimum possible
label in $S_3$ and $S_4$ is $5$. Also, the maximum label on $S_1$ and $S_2$ is 4 . Therefore, we cannot properly label the central
vertex and the theorem follows.
\qed
\end{proof}

\begin{figure}
\centering
\includegraphics[scale=0.95]{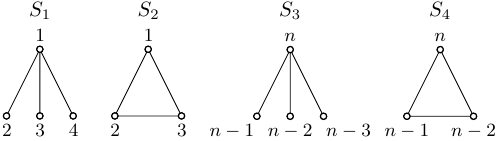}
\caption{The subgraphs from the proof of Theorem \ref{thm:fan}.}
\label{fig:vejar-sit}
\end{figure} 

\begin{figure}
\centering
\includegraphics[scale=0.95]{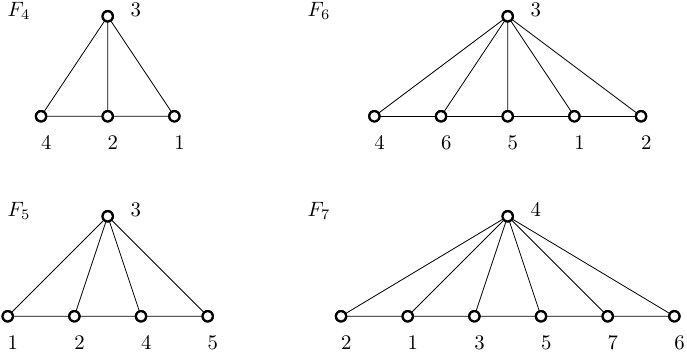}
\caption{Canonical ESD labelings for $F_4, F_5, F_6$, and $F_7$.}
\label{fig:vejar_male}
\end{figure} 

\subsection{Complete bipartite graphs}

We need to introduce a notion of isomorphism for vertex labelings.

\begin{definition}
	Vertex labelings $\phi_1$ and $\phi_2$ on $G$ are isomorphic if there
	exists an automorphism $f$ of $G$ such that $\phi_1(v) = \phi_2(f(v))$ for
	every $v \in V(G)$.
\end{definition}

We will prove the following theorem, covering all cases for complete
bipartite graphs.

\begin{theorem} \label{thm:kqr}
	Let $K_{p,q}$ be a complete bipartite graph on $n=p+q$ vertices, $p \leq q$, then the following holds.
	\begin{enumerate}
		\item For $p, q > 2$ there is no ESD labeling on $K_{p,q}$.
		\item If $p = 2$, then there exists exactly one possible ESD labeling
		up to isomorphism.
		\item If $p = 1$, then every canonical labeling is an
		ESD labeling.
	\end{enumerate}
\end{theorem}

\begin{proof}
\begin{enumerate}
	\item Suppose for a contradiction that we have some
	canonical ESD labeling $\phi$. Denote the parts of $K_{p,q}$ by $P$ and $Q$.
	We will divide the proof into two cases.
	\begin{itemize}
		\item There exist two vertices $v_1, v_2,$ in $P$
		 and two vertices $u_1,u_2$ in $Q$ such that
		$\phi(v_2) = \phi(v_1) +1$ and $\phi(u_2) = \phi(u_1)+ 1.$ 
		Then $w_\phi(v_1u_2) = w_\phi(v_2u_1)$, and we get a contradiction.

		\item There exists a part (without loss of generality $P$) such that $\phi(v_1) \neq \phi(v_2)+1$ for every $v_1,v_2 \in P$.
		Since $P$ is of size at least 3, there exist two vertices $v'_1, v'_2 \in P$ with labels smaller than $n$. Thus there exists a vertex $u_1 \in Q$ with label $\phi(v'_1)+1$ and $u_2 \in Q$ with label $\phi(v'_2)+1$.
		Then $w_\phi(v'_1u_2) = w_\phi(v'_2u_1)$, a contradiction.
	\end{itemize}

	\item We denote the vertices of the part of the size 2 as $v_1,v_2$. The vertices of the other part will be $u_1,\ldots,u_q$. Let $\psi$ be 
	a vertex labeling of $K_{2,q}$ defined as follows:
	\begin{itemize}
		\item $\psi(v_1) = 1$,
		\item $\psi(v_2) = n$,
		\item $\psi(u_i) = i+1$ for $i \in \{1,\ldots,q \}$.
	\end{itemize}
	Observe that $\psi$ is indeed a canonical ESD labeling.
	For $q = 2$, one can easily check that this is the only
	canonical ESD labeling up to isomorphism.

	Now, for a contradiction, assume that a canonical ESD
	labeling $\psi'$, nonisomorphic to $\psi$, exists.
	Furthermore, $n > 4$, and we can assume that $\psi'(v_1) < \psi'(v_2)$.
	Either $\psi'(v_1) \neq 1$ or $\psi'(v_2) \neq n$.
	We distinguish two cases.
	\begin{enumerate}
		\item It holds that $\psi'(v_2) = \psi'(v_1) + 1$.

		Since $n > 4$, we can find two vertices $a_1,a_2$ in the other part
		such that $\psi'(a_2) = \psi'(a_1) + 1$. Similarly as in case (1)
		of this theorem, $w_{\psi'}(v_1a_2) = w_{\psi'}(v_2a_1)$ and we get a contradiction.

		\item It holds that $\psi'(v_2) \neq \psi'(v_1) + 1$.

		Then there exist two distinct vertices $u_j,u_k \in \{u_1,\ldots,u_q \}$
		such that one of the following holds.
		Either $\psi'(u_j) = \psi'(v_1) + 1$ and $\psi'(u_k) = \psi'(v_2) + 1$, or
		$\psi'(u_j) = \psi'(v_1) - 1$ and $\psi'(u_k) = \psi'(v_2) - 1$.
		In both cases $w_{\psi'}(v_1u_k) = w_{\psi'}(v_2u_j)$ and we are done.
	\end{enumerate}
	We conclude that no such $\psi'$ exists.

	\item Every edge in a canonical labeling of $K_{1,q}$ has a unique sum since every edge is incident to the central vertex of degree $q$.
\end{enumerate}
\qed \end{proof}

We note that the first part of Theorem \ref{thm:kqr} can be proved
by using Theorem  \ref{thm:nec} but we think that our proof is more clear.

\subsection{Trees}

We already showed that paths and stars are ESD graphs. The following theorem
solves the general case of trees.

\begin{theorem} \label{thm:trees}
Every tree has a canonical ESD labeling.
\end{theorem}

\begin{proof}
Let $T$ be an $n$-vertex tree with root in $v_1 \in V(T)$.
We will denote by $v_1,\ldots,v_n$ an ordering of vertices visited
in a breadth-first search on $T$, starting in $v_1$.
We define a labeling $\phi$ as $\phi(v_k) := k, \forall v_k \in V(T)$. We
want to show that $\phi$ is a canonical ESD labeling.

Consider some vertex $v_i, i>1,$ and its parent $v_j$. Denote by $T'$ the
tree induced by vertices $v_1,\ldots,v_{i-1}$. See Figure \ref{fig:tree}
for an illustration. We claim that the following
holds:
$$w_\phi(v_iv_j) > w_\phi(v_av_b),\, \forall v_av_b\in E(T').$$

By the level of a vertex we mean its distance to root vertex $v_1$.
Without loss of generality, assume that $a < b$. We distinguish these cases.
\begin{itemize}
	\item The edge $v_av_b$ has both endpoints on a level lower or equal
	to the level of $v_j$. Then $a < j$ and $b < i$ and from that $a+b < i+j$.
	\item If $v_a = v_j$, then $v_j$ is the common parent of $v_b$ and $v_i$.
	Thus $b < i$ and from that $b+j < i+j$.
	\item The vertex $v_a$ is on the same level as $v_j$ and $v_a \neq v_j$.
	Then $a < j$ and $b < i$, implying that $a+b < i+j$. 
\end{itemize}

We proved the claim and the theorem follows.
\qed
\end{proof}

\begin{figure}
\centering
\includegraphics[scale=0.7]{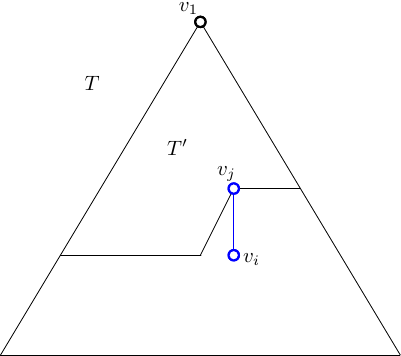}
\caption{An illustration of situation in Theorem \ref{thm:trees}.}
\label{fig:tree}
\end{figure} 

\subsection{Cycles}

\begin{theorem} \label{thm:cycles}
	Every cycle graph $C_n$ is an ESD graph.
\end{theorem}

\begin{proof}
Let us denote the vertices of $C_n$ as $v_1,\ldots,v_n$ in a circular order. We
distinguish two cases:
\begin{enumerate}
	\item If $n$ is even, then we assign labels as follows:
	\begin{itemize}
		\item $\phi(v_i) = i$ for all $i \in \{1,\ldots,n-2\}$,
		\item $\phi(v_{n-1}) = n$,
		\item $\phi(v_n) = n - 1$.
	\end{itemize}
	Weights of the edges $v_iv_{i+1}$ for $i \in \{1,\ldots,n-3\}$
		are odd integers $3,5,\ldots,2n-5$. The weight of the edge
		$v_{n-1}v_n$ is $2n-1$ and therefore is odd as well. The remaining
		edges will be even; $w_\phi(v_nv_1) = n$ and $w_\phi(v_{n-2}v_{n-1}) = 2n-2$.
		We conclude that the edge-weights are unique.

	\item If $n$ is odd we assign labels as follows:
	\begin{itemize}
		\item $\phi(v_i) = i$ for all $i \in \{1,\ldots,n\}$.
	\end{itemize}
		The weights of the edges between $v_iv_{i+1}$ for $i \in \{1,\ldots,n-1\}$
		will be odd integers $3,5,\ldots,2n-1$. The weight of the edge
		$v_1v_n$ is equal to $n+1$ and therefore it is even. Again, all edge-weights
		are unique and we get a canonical ESD labeling.
\end{enumerate}
\qed
\end{proof}

\subsection{Generalized sunlet graphs}

We recall that a graph is \emph{unicyclic} if it contains exactly one cycle.

\begin{definition}
 A \emph{generalized sunlet graph} $S_k^p$ is a unicyclic graph obtained by taking a cycle graph
 $C_k$, with $V(C_k) = \{c_1,\ldots,c_k \}$, and joining path graphs $R_i, i \in \{1,\ldots,k \}$ of order $p$ to this cycle so that
 one of the endpoints of $R_i$ is identified with~$c_i$.
\end{definition}

\begin{theorem}
Let $S_k^p$ be a generalized sunlet graph.
If $k$  is odd and $p$ is even, then $S_k^p$ has a canonical ESD labeling. 
\end{theorem}
\begin{proof}
We denote the vertices of $S_k^p$ in the following way.
\begin{itemize}
\item Vertices on the cycle are $v_1, v_{p + 1}, v_{2p+1}, \ldots, v_{(k-1)p+1}$.
\item Vertices on the path joined to the vertex $v_{ip+1}$ are
consecutively\\ $v_{ip+1}, \ldots, v_{(i+1)p}$, for $1 \leq i \leq k$.
\end{itemize}

We define a vertex labeling $\phi$ as $ \phi(v_i) := i$. We claim that $\phi$
is a canonical ESD labeling. All edge-weights on attached paths are odd,
because we get them as a sum of two consecutive numbers. Furthermore, all
edge-weights on a path joined to vertex $v_{ip+1}$ are smaller than
edge-weights on a path joined to vertex $v_{(i+1)p+1}$. Thus all edge-weights on
paths are distinct.  All edge-weights on the cycle expect for the edge
$v_1v_{(k-1)p+1}$ are in the form $k'p+2$ for some $k' \in \N$. Thus they are all even
and distinct.

It remains to show that the edge $ v_1v_{(k-1)p+1}$ has an edge-weight
different from all others. For a contradiction we assume that the edge-weight
$(k-1)p+2$ was already used. It is even, so it can be only used on the cycle.
Thus, $k-1$ must be a sum of two distinct consecutive natural numbers. That
gives a contradiction, because $k-1$ is even.
\qed
\end{proof}

For the other parity conditions we were not able to prove that there is always
an ESD labeling. Thus we leave as an open problem to determine if all
generalized sunlet graphs have a canonical ESD labeling. Small examples
suggest that it might be true.

\begin{theorem}
Let $S_k^p$ be a generalized sunlet graph.
If $k$ and $p$ are odd or $k$ is even and $p$ is odd or even, then $S_k^p$ has an ESD labeling with label set $L$ of size $(p+1)k - 2$.
\end{theorem}
\begin{proof}
In both cases of parity of $k$, the unique cycle in $S_k^p$ will be labeled in
the same way as in Theorem \ref{thm:cycles}. Observe that the greatest edge-
weight on the edges of cycle is $2k-1$.

The rest of the vertices is labeled by the following procedure. Start with
label $i := 2k-1$ and label by $i$ an unlabeled vertex which is adjacent to
the vertex with the minimum label. Increment $i$ by one and repeat the step.
We see that in every step we get one new edge-weight. Furthermore, this
edge-weight is always greater than any previous edge-weight created during this
procedure and all these edge-weights are greater than any edge-weight on an
edge in the cycle. Thus, the resulting labeling is ESD. Furthermore, we labeled
the cycle with $k$ labels with $1,\ldots,k$ and then the remaining $p(k-1)$ vertices with labels
$2k-1,\ldots, (p-1)k - 2$. This implies that the set of labels $L$ is of size $(p-1)k-2$.
\qed
\end{proof}

\subsection{Grids}

\begin{definition}
A $k\times l$ grid graph $G_{k,l}$ is the Cartesian product of path graphs $P_k$ and $P_l$.
\end{definition}

\begin{theorem}
Let $G_{k,l}$ be a grid graph. If $k$ or $l$ is even then $G_{k,l}$ has a canonical ESD labeling.
\end{theorem}

\begin{proof}
Without loss of generality assume that $k$, the number of
columns, is even.  Let us denote the vertices in the $i$-th row by  $v_{(i-1)k+1},
\ldots, v_{ik}$ for every $i \in \{1, \ldots, l\}$.  We define a canonical vertex
labeling $\phi$ as $\phi(v_i) := i$. We want to show that $\phi$ is an
 ESD labeling on $G_{k,l}$.

The graph $G_{k,l}$ with labeling $\phi$ has the following edge-weights:
\begin{itemize}
\item $2(i-1)k + 3, \ldots, 2ik-1$ in the $i$-th row for every $1 \leq i \leq l$,
\item $2ik + 3, \ldots, 2(i+1)k-1$ in the $(i+1)$-th row for every $0 \leq i \leq l-1$,
\item $(2i-1)k+2, \ldots, (2i+1)k$ on edges between the $i$-th and the $(i+1)$-th row $1 \leq i \leq l-1$. 
\end{itemize}
All edge-weights on rows are odd and all edge-weights in the $i$-th row are smaller than all edge-weights in the $(i+1)$-th row. A similar argument holds for all edge-weights in columns. This concludes the proof.
\qed
\end{proof}

\subsection{Complete graphs}

From Theorem \ref{thm:nec} it is clear that complete graphs $K_n$ for $n > 3$
do not have a canonical ESD labeling. However, the following theorem provides
a simple way how to find an ESD labeling. We recall that \emph{Fibonacci sequence}
is defined as $F_0 := 0, F_1 := 1$, and $F_n = F_{n-1} + F_{n-2}$ for $n > 1$.
We note that the following theorem implies that for any $n$-vertex graph,
$F_{n+1}$ labels suffice to construct an ESD labeling.

\begin{theorem}
	There exists an ESD labeling with $F_{n+1}$ labels for every complete graph $K_n$.
\end{theorem}
\begin{proof}
	For given $K_n$ we label its vertices $\{v_1, \ldots, v_n\}$ by function
	$\phi_n$, defined as $\phi_n(v_i) := F_{i+1}$.

	We show that this is an ESD labeling by induction. We see that for $K_1$
	and $K_2$, $\phi_1$ and $\phi_2$ are clearly ESD labelings. Now we want to
	prove that $\phi_n$ is an ESD labeling. We see that $v_1,\ldots,v_{n-1}$
	are labeled as in $\phi_{n-1}$. The largest possible sum on an edge in
	$\phi_{n-1}$ is $F_n+F_{n-1}$. The only new label in $\phi_n$ is $F_{n+1}$
	and the minimum possible sum on an edge incident with $v_n$ is
	$F_{n+1}+F_2 = F_n+F_{n-1}+1$. Thus, assuming that $\phi_{n-1}$ is an ESD
	labeling, $\phi_n$ is an ESD labeling as well.
\qed
\end{proof}

\section{Games with ESD labeling}

Tuza in his paper \cite{tuza2017graph} emphasized that only few papers on graph labeling games exist. He defined a new game from ESD labeling. 

\begin{definition}
	We call a vertex of graph \emph{free} if it is not labeled yet.
\end{definition}

\begin{definition}
Let $G=(V,E)$ be a graph and $L=\{1,\ldots,l \}$ its set of labels. Alice and Bob are two players
who alternate after every move. Alice starts. In each move,  player chooses a free vertex of $G$ and assigns
an unused label to it. The move is \emph{legal} if the resulting edge-weights are unique.

The game ends if there is no legal move possible or an ESD labeling is created. Alice wins if
an ESD labeling is created, otherwise Bob wins.

We say that an ESD labeling game is \emph{canonical} on $G$ if $|L| = |V(G)|$.
\end{definition}

One can also define other variants of this game. For example, Bob can be the starting one.
Also, our definition of game is a Maker-Breaker type of game, but it is possible to define
Achievement and Avoidance type of this game as well.

\begin{proposition} \label{thm:bob}
If a graph $G$ does not have a canonical ESD labeling then Bob has a winning strategy in
the canonical game on $G$.
\end{proposition}

\begin{proof}
If a graph $G$ does not have a canonical ESD labeling then Alice can not make any canonical ESD labeling and Bob eventually wins.
\qed
\end{proof}

\begin{theorem}
Alice wins every canonical game on a star $S_n$.
\end{theorem}

\begin{proof}
We already proved in Theorem \ref{thm:kqr} that every canonical vertex labeling on a star graph is edge-sum distinguishing. Thus Alice wins every game regardless on the course of the game.
\qed
\end{proof}

\begin{theorem}
Bob wins every canonical game on
a complete bipartite graph $K_{p,q}$,  $p \leq q$, where $p = 2$. 
\end{theorem}

\begin{proof}
We recall Theorem \ref{thm:kqr}. The graph $K_{p,q}$,  $p \leq q$, where $p = 2$, needs
to have labels $1$ and $p+q$ on the smaller part. Thus a winning strategy for Bob is
to assign a label $w$, such that $1 < w < p+q$, on a free vertex of the smaller part in
his first move. Now it is not possible to build a canonical ESD labeling and Bob wins.
\qed
\end{proof}

Tuza also asked \cite[Problem 3.1]{tuza2017graph} the following question:
Given $G=(V, E)$, for which values of $l$ can Alice win the edge-sum
distinguishing labeling game? We partially answer this question by the
following theorem.

\begin{theorem} \label{thm:odhad}
Let $G$ be a graph, $\Delta$ its maximum degree, and $L$ its set of labels. If $|L| \geq (\Delta^2+1)n + \Delta {n-1 \choose 2}$, then Alice has a wining
strategy.
\end{theorem}

\begin{proof}	
For each vertex $v$ of $G$, define a set $S_v$ as the set of labels available for $v$.
In the beginning of every game, $S_v = L$ for every $v \in V(G)$.

Our goal is to build a winning strategy for Alice. In $k$-th move, a player
assigns to a free vertex $v$ some label $\phi(v) \in S_v$. We update
the set of labels in the following three steps right after the player's choice.

\begin{enumerate}
	\item We delete $\phi(v)$ from $S_u$ for every $u \in V(G)$. This label
	cannot be used twice, since ESD labeling is an injective mapping.
	\item For every free vertex $y$ incident to $v$ we delete all labels
	$l_{y,e}$ such that $l_{y,e} + \phi(v) = w_\phi(e)$ for some
	edge $e$ with both endpoint vertices labeled and incident with $v$. In this process, we delete at most ${k-1 \choose 2}$ labels
	from $S_y$.
	\item For every free vertex $z$ and for every vertex $z' \in N(z)$,
	such that $z'$ is already labeled, we delete from $S_z$
	all labels $l'$ such that
	$$l' + \phi(z') = w_\phi(vv'),\,\,\forall v' \in N(v).$$
	Within these steps, we delete at most $\Delta^2$ labels from label set of
	every free vertex.
\end{enumerate}
If the label set for every free vertex is nonempty before every move, Alice
wins. Let us count how many labels are deleted in course
of the game for every free vertex.
\begin{itemize}
	\item We delete at most $n-1$ labels through all first steps.
	\item We delete at most $\Delta {n-1 \choose 2}$ labels through all second
	steps.
	\item We delete at most $\Delta^2n$ labels in third steps.
\end{itemize}
Summarized, we delete at most $ (\Delta^2+1)n + \Delta {n-1 \choose 2} - 1$ labels.
If we have one extra label available, we can always find a label for a free vertex
and our bound is proved. Note an important fact that it does not matter how Bob
plays and the resulting labeling is ESD.
\qed
\end{proof}	

Observe that this theorem also gives us a bound on the size of label set for general
graphs. This follows by taking Proposition \ref{thm:bob} into account.

Also, by a similar analysis, one can obtain the following theorem for path graphs.

\begin{theorem}
	Let $P_n$ be a path graph on $n$ vertices. If $|L| \geq 5n$, then Alice
	wins every game on $P_n$.
\end{theorem}

\section{Concluding remarks}

We studied a new type of graph labeling, introduced by Tuza, which  is similar
to magic (and antimagic) labelings, harmonious labelings and has a relation to
the Sidon sequences. We would like to highlight our main results.

\begin{itemize}
	\item We proved that trees, cycles and complete bipartite graphs with one part of size 2 have a canonical ESD labeling.
	\item We proved that in some cases grid graphs and generalized sunlet graphs
	do have a canonical ESD labeling.
	\item We showed that
	fan graphs and complete bipartite graphs with both parts of size at least 3 do not have a canonical ESD labeling.
	\item We studied a Maker-Breaker type of game, applied our previous results and
	derived a general bound on number of labels such that Maker wins the game.
\end{itemize}

\paragraph{Open problems.} Aside from Tuza's original game-oriented
problems proposed in \cite{tuza2017graph}, we emphasize the following question,
arising from the results in this paper.

\begin{problem}
What is the maximum possible number of edges for $n$-vertex connected graphs
so that every graph with such number of edges has a canonical ESD labeling?
\end{problem}

From Theorem \ref{thm:trees} we see that to answer this question one needs to
resolve the case of unicyclic graphs which is now only partially solved.

\section*{Acknowledgment}

Both authors were supported by the grant SVV–2017–260452. Jan Bok was
supported by the Center of Excellence - ITI (P202/12/G061 of GA\v{C}R). Nikola
Jedličková was supported by the Student Faculty Grant of the Faculty of
Mathematics and Physics, Charles University.

Both authors would like to thank Robert Šámal for his feedback and suggestions
regarding the paper. We would also like to thank the anonymous referee for
valuable advices that led to substantial improvements of the paper.

\bibliographystyle{plain}
\bibliography{bibliography}

\end{document}